\documentclass[11pt, letterpaper]{amsart}
\usepackage{graphicx}
\usepackage{amsbsy,amsmath,amsfonts,amssymb,amscd}
\usepackage[utf8]{inputenc}
\usepackage{amssymb}

\usepackage[hyperfootnotes=false, colorlinks=true, linkcolor=blue]{hyperref}
\usepackage{mathrsfs}
\usepackage{epsfig, graphicx}
\usepackage{graphics,epsfig}
\usepackage{mathrsfs}
\usepackage{amsthm}
\usepackage{mathtools}
\usepackage[dvipsnames]{xcolor}
\usepackage{changes}
\usepackage[margin=1in]{geometry}
\DeclareUnicodeCharacter{211D}{$\mathbb{R}$}

\title{Quantitative Stability of Aggregation-Diffusion under Local Attraction}
\author[Yunhong Deng]{Yunhong Deng}
\address{Department of Mathematics, Statistics and Computer Science, University of Illinois, Chicago, Illinois 60607, USA}
\email{ydeng43@uic.edu}

\date{October 7, 2026}

\newtheorem{theorem}{Theorem}
\newtheorem{lemma}[theorem]{Lemma}

\newtheorem{proposition}[theorem]{Proposition}

\theoremstyle{definition}
\newtheorem{definition}{Definition}
\theoremstyle{definition}
\newtheorem{remark}{Remark}
\theoremstyle{definition}

\theoremstyle{definition}
\newtheorem{assumption}{Assumption}

\newcommand{\ud}{\,\mathrm{d}}

\begin{document}
\begin{abstract}
In this note, we establish the quantitative stability of nonlinear aggregation-diffusion energy for compactly supported steady states. Our consideration focuses on the degenerate diffusion with local attraction that are asymptotically flat at infinity, such as Riesz kernels. Our main result characterizes the spectral gap of the Wasserstein Hessian at steady states.
\end{abstract}

\maketitle
\section{Introduction}
We consider the aggregation-diffusion energy defined for a probability density $\varrho$ given as below,
\begin{equation}\label{0}
    \mathcal{F}(\varrho)
    =
    \underbrace{\frac{1}{m-1} \int_{\mathbb{R}^d}\varrho^m\ud x}_{\text{internal energy}}
    +
    \underbrace{\frac12
    \int_{\mathbb{R}^d}
    (W \ast \varrho) \times \varrho  \ud x}_{\text{interaction energy}}.
\end{equation}
We focus on the diffusion-dominate regime \cite{carrillo2018ground, delgadino2022uniqueness} with $m \ge 2$, and a radial increasing local attraction kernel $W$. This free energy functional is closely related to the following aggregation-diffusion equation \cite{carrillo2019aggregation} which describes the collective particle dynamics under two competing effects of the internal and interaction energies:
\begin{equation}\label{0-1}
    \frac{\partial \varrho}{\partial t} - \Delta \varrho^{m} - \mathrm{div}(\varrho \nabla W \ast \varrho) = 0 \text{ on } \mathbb{R}^{d},
\end{equation}
since seminal works near the incipient twenty-first century \cite{JKO, otto2001geometry, carrillo2003kinetic}.

These seminal works gave rise to what is now known as the gradient flow formulation in Wasserstein spaces $\mathscr{P}_{2}(\mathbb{R}^{d})$, with the the 2-Wasserstein distance
\begin{equation*}
    W_{2}(\mu, \omega) = \inf \bigg\{\int_{\mathbb{R}^d} |S(x) - x|^2 \ud x : S\# \mu = \omega\bigg\},
\end{equation*}
where $\#$ is the push-forward operator. For preliminaries on optimal transport, we refer the reader to Section~\ref{sec:2-1-optimal} below. In this gradient flow formulation, one regards~\eqref{0-1} as continuity equation along the negative Wasserstein gradient of the energy functional~\eqref{0}
\begin{equation}\label{0-2-gradient}
    \frac{\partial \varrho}{\partial t} - \mathrm{div}(\varrho D_{W}\mathcal{F}(\varrho)) = 0 \text{ on } \mathbb{R}^{d},\quad D_{W}\mathcal{F}(\varrho) = \nabla \bigg[\frac{m}{m - 1} \varrho^{m - 1} + W \ast \varrho\bigg].
\end{equation}
At the time-discrete level, this gradient flow structure can be either seen from the gradient descent (forward Euler scheme)
\begin{equation*}
    \varrho(\cdot + \tau) \approx \big(\mathrm{id} -\tau D_{W} \mathcal{F}(\varrho(t))\big) \# \varrho(t)
\end{equation*}
where the continuity equation can be regarded as the infinitesimal push-forward \cite[Proposition 8.1.8]{ambrosio2008gradient}, or the minimizing movement scheme \cite{JKO} (backward Euler scheme)
\begin{equation*}
    \varrho(\cdot + \tau) \approx \arg\min \mathcal{F}(\varrho) + \frac{1}{2\tau}W_2(\varrho, \varrho(t))^2.
\end{equation*}
And the functional~\eqref{0} act as the Lyapunov functional of the dynamical system~\eqref{0-1} with a form of energy dissipation given as follows,
\begin{equation*}
    \frac{\mathrm{d}\mathcal{F}(\varrho)}{\mathrm{d}t} = -\int_{\mathbb{R}^d} |D_W \mathcal{F}(\varrho)|^2 \ud \varrho \le 0
\end{equation*}
where $\varrho$ is the solution to the aggregation-diffusion flow~\eqref{0-1}.

\subsection{Classical models} Gradient flow formulation in Wasserstein spaces was originally introduced to analyze the long-time behavior of porous media diffusion \cite{otto2001geometry}. For convex attraction kernels\footnote{If $W$ is convex and non-constant, then it is forced to grow to infinity when $|x| \to \infty$.} such as $W(x) = |x|^3$, the long-time behavior of the aggregation-diffusion system~\eqref{0-1} has been well-understood following the seminal work \cite{carrillo2003kinetic}, in which Carrillo--McCann--Villani proved that, for all $\varrho \in \mathscr{P}_{2, \mathrm{ac}}(\mathbb{R}^{d})$,
\begin{equation}\label{1-2-transport}
    \frac{\mathrm{d}^{2}}{\mathrm{d}t^2}\Big|_{t = 0}\mathcal{F}\big((\mathrm{id} + tX)\# \varrho\big) \ge \beta \|X\|^{2}_{L^{2}(\mathbb{R}^{d},\, \varrho)},\quad \text{ for all } X \in T_{\varrho}\mathscr{P}_{2}(\mathbb{R}^{d}) \text{ and } \int_{\mathbb{R}^{d}} X\ud \varrho = 0
\end{equation}
for some positive constant $\beta$, based on the displacement convexity of the energy functional \cite{mccann1997convexity} (also see \cite[Chapter 20]{villani2008optimal}). Specifically, when $W$ is strongly convex, $\beta$ is given by the lower bound of $D^{2}W$. $T_{\varrho}\mathscr{P}_{2}(\mathbb{R}^{d})$ is the tangent space of the Wasserstein space (see Definition~\ref{def_4_tangent} below). Building on the second-variation, Carrillo--McCann--Villani \cite{carrillo2003kinetic} concluded the exponential convergence towards steady states, buiding on the following transport inequality and log-Sobolev inequality:
\begin{equation*}
    \inf_{a \in \mathbb{R}^{d}} \frac{\beta}{2} W_{2}(\tau_{a}\omega, \mu)^{2} \le \mathcal{F}(\omega) - \mathcal{F}(\mu) \le \frac{1}{2\beta} \big\|D_{W}\mathcal{F}(\omega)\big\|_{L^{2}(\mathbb{R}^{d},\, \omega)}^{2},\quad \text{for all $\omega$ in $\mathscr{P}_{2, \mathrm{ac}}(\mathbb{R}^{d})$,}
\end{equation*}
where $\mu$ is a minimizer of the aggregation-diffusion energy, and $\tau_{a}$ denotes the translation by vector $a$. It is notable that the aggregation-diffusion energy~\eqref{0} is invariant under the action of translation. Hence, one can find a family of minimizers $\{\tau_{a}\mu\}_{a \in \mathbb{R}^d}$, and the minimizer $\mu$ is called integrable (see, for example, \cite{carlotto2015slowly, chodosh2022riemannian}).

Extensive study interest afterwards was paid to non-convex kernels $W$ such as Riesz and Newton kernels, defined as the inverse of fractinoal Laplacian $W\ast \varrho := (\Delta)^{-\alpha/2}\varrho$ for some positive exponent $\alpha$. A standard model investigated extensively is the Patlak-Keller-Segel equation \cite{blanchet2008convergence, blanchet2008infinite, bedrossian2010intermediate, bedrossian2011global, kim2012patlak, carlen2013stability} on $\mathbb{R}^{2}$ defined as
\begin{equation*}
    \frac{\partial \varrho}{\partial t} - \Delta \varrho - \frac{1}{2\pi}\mathrm{div}(\varrho \nabla (\log|x| \ast \varrho)) = 0.
\end{equation*}
It can be seen as the Wasserstein gradient flow of the following energy
\begin{equation*}
    \mathcal{F}(\varrho) = \int_{\mathbb{R}^{2}} \varrho \log(\varrho) \ud x + \frac{1}{4\pi}\int_{\mathbb{R}^{2}} (\log|x| \ast \varrho) \times \varrho \ud x.
\end{equation*}
This model admits a well-known $8\pi$ critical mass phenomenon (see, for example, \cite[Section 2.1]{carrillo2019aggregation}): For solution started with mass less than $8 \pi$, its behavior is dominated by the diffusion part, where the solution $\varrho(t)$ will spread out, and converge weakly to $0$ as $t \to \infty$. For solution started with mass larger than $8 \pi$, its behavior is dominated by the aggregation part, where the solution $\varrho(t)$ will blow-up in finite time. Something interesting will happen in the $8\pi$ critical mass regime, where the two forces balance with each-other, and one can get a global existence of smooth solutions. In this case, there exists stationary solutions
\begin{equation*}
    \varrho_{a, \beta}(x) = \frac{8\beta^2}{(|x - a|^2 + \beta^2)^2}
\end{equation*}
and this class is closed under translation and dilation. One recent work by Buseghin--Collot \cite{buseghin2026determination} performs a delicate calculation in the critical mass regime, where they showed, for initial value with finite second momentum, the solution will blow-up at $\infty$-time.

\subsection{Background}
To understand the long-time behavior of a gradient flow system~\eqref{0-1}, a first step would be investigating the landscape of the energy functional $\mathcal{F}(\varrho)$ on $\mathscr{P}_{2}(\mathbb{R}^{d})$.

There are rich results supporting the existence of steady states, and their qualitative properties. In \cite{carrillo2016nonlinear}, Carrillo--Hittmeir--Volzone--Yao established their well-known radial-symmetry result of steady states, and conclude the uniqueness\footnote{Hereinafter, the uniqueness refers to being unique up to translation.} of steady-states building on \cite{carrillo2015ground, kim2012patlak}. The uniqueness of steady-states implies the long-time convergence of a nonlinear Keller-Segel model with Newton kernel in $2$-dimension without relying on geodesic convexity. In \cite{carrillo2018ground}, Carrillo--Hoffmann--Mainini--Volzone later considered ground steady states of~\eqref{0} with degenerate diffusion and Riesz kernels in the diffusion dominant regime. They proved the existence of minimizers, and characterized them as radial decreasing, compactly supported and smooth inside their support interior, and established their uniqueness in one-dimension for $m \ge 2$.

Carrillo--Delgadino--Patacchini \cite{carrillo2019existence} provided the existence of ground steady states for general aggregation-diffusion energies. Compared to degenerate diffusion where $m > 1$, they addressed the linear and fast diffusion where $m \le 1$, and in this case, the diffusion effects is harder to confine. Furthermore, they proved a logarithmic growth of attraction kernel is sharp to confine the linear diffusion. The general spirit in their proof for linear and fast diffusion is a compactness of energy sub-level set due to the concentration (see, for example, Lion's concentration compactness principal \cite{lions1984concentration, lions1984concentration-2}).

The current work is based on the result of Delgadino--Yan--Yao \cite{delgadino2022uniqueness}, in which they proved the uniqueness of steady states of aggregation-diffusion energy for a wide-class of attraction-kernels in the diffusion dominated regime when $m \ge 2$, and constructed infinite many solutions for $1 < m < 2$ for selected kernels.

Let us summarize their assumption of $W$, which is also our starting assumption. Their method in proving uniqueness is a height interpolation of radial decreasing densities, along which (together with the following assumptions) they showed that aggregation-diffusion energy~\eqref{0} is convex. They then concluded the uniqueness as the corollary of the convexity.

\begin{assumption}[attraction kernel]\label{assumption}We assume that the attraction kernel satisfies the following.
\begin{enumerate}
    \item $W(x)$ is radial increasing, in the sense that $W(x) = W_{\mathrm{ra}}(|x|)$ for some (strictly) increasing function $W_{\mathrm{ra}} : (0, \infty) \to \mathbb{R}$.
    \item $W(x)$ is $C^{\infty}$ on $\mathbb{R}^{d} - \{0\}$.
    \item $W$ is no more singular than the critical Riesz kernel in the sense that $W_{\mathrm{ra}}^{\prime}(r) \le C_{W} r^{k - 1}$ for all $0 < r < 1$ for some $k > -d$ and $C_{W} \ge 0$.
    \item There exists a constant $C_{W}$ such that $W^{\prime}_{\mathrm{ra}}(r) \le C_{W}$ for all $r \ge 1$.
    \item $W_\mathrm{ra}(r)$ satisfies the doubling hypothesis
    \begin{equation*}
        W_{\mathrm{ra}}(a+ b)_+ \le C_{W} (1 + W_{\mathrm{ra}}(1 + a)_+ + W_{\mathrm{ra}}(1 + b)_+),
    \end{equation*}
    for all $a, b > 0$ with a constant $C_{W} \ge 0$, where $W_{\mathrm{ra}}(a)_+$ is the positive part.
\end{enumerate}
\end{assumption}

Under these assumptions, Delgadino--Yan--Yao \cite{delgadino2022uniqueness} provided the following characterization of equilibrium, combing the results in \cite{carrillo2016nonlinear, carrillo2018ground}.
\begin{theorem}\label{proposition-3-steady}For the aggregation-diffusion energy~\eqref{0} with $m \ge 2$, the minimizer $\mu$ centered at $0$ satisfies the following:
\begin{enumerate}
    \item $\mu$ is continuous, compactly supported, symmetric, smooth inside support interior
    \item $\mu$ is strictly radial decreasing in the support interior, and positive inside support interior
    
    \item $\mu$ is non-degenerate at free boundary in the sense that $\mu(x) \sim \delta(x, \mathbb{R} - \Omega)^{\frac{1}{m - 1}}$ where $\Omega$ is the support of $\mu$, and $\delta$ is the distance.
\end{enumerate}
\end{theorem}
\begin{remark}
As noted in \cite[Section 1.2]{delgadino2022uniqueness}, assumptions (3) and (4) in Assumption~\ref{assumption} are used only to ensure all steady states are radially decreasing and compactly supported (Theorem~\ref{proposition-3-steady}), these two assumptions play no role in their uniqueness proof.
\end{remark}

\begin{assumption}\label{assumption-1}We additionally assume the following:
\begin{enumerate}
    \item Non-degeneracy: $\lim_{r \to 0} W^{\prime }_{\mathrm{ra}}(r)/r > 0$.
    \item Regularity: $|W^{\prime \prime}_{\mathrm{ra}}(r)| \lesssim r^{k - 2} \text{ for } 0 < r \le 1$.
\end{enumerate}
The non-degeneracy assumption rules out kernels $W$ that are degenerately convex at origin, which reduces technical difficulties.\footnote{However, degenerately convex kernels constitute an important class of models, which was considered in \cite[Section 4.3]{carrillo2003kinetic}. Removing the non-degeneracy assumption will be the future direction.} The regularity assumption is only used to prove the self-adjointness of the Wasserstein Hessian in Theorem~\ref{theorem:2_spectrum}.

These assumptions interoperate the Riesz kernels $W(x) = |x|^k/k$ with $-d < k < 0$.
\end{assumption}

A natural question, following their uniqueness result, is that how stable are these steady-states. For example, whether the energy gap
\begin{equation*}
    \mathcal{F}(\omega) - \mathcal{F}(\mu)
\end{equation*}
controls a distance between $\omega$ and $\mu$. Here a finite dimensional analogy is the Lojasiewicz distance inequality which characterizes a stable critical point $x_{0}$ of a real analytical function $f$ using (see, for example, \cite{feehan2019resolution}),
\begin{equation*}
    f(x) - f(x_{0}) \ge \beta |x - x_{0}|^{\theta} \text{ for } x \in B_{\sigma}(x_{0}),
\end{equation*}
with positive $\beta$, $\sigma$, and $\theta \ge 2$, where the Lojasiewicz exponent $\theta$ reflects the degree of degeneration of the critical point (the Adam-Simon positivity condition, see for example \cite{carlotto2015slowly}). This Lojasiewicz distance inequality together with the counterpart gradient inequality are critical in uniqueness and quantitative stability problems in geometric analysis (for example, \cite{colding2015mean, engelstein2022quantitative, chodosh2022riemannian, choi2025wasserstein, deng2026phase}).

It would be important that such an quantitative stability result comes from an optimal transport characterization, because under this setting, the stability would be closely related to quantitative convergence of the aggregation-diffusion flow~\eqref{0-1}, and the quantitative convergence has been questioned in various places, for example, \cite{carrillo2016nonlinear, carlen2025stability}.

However, the quantitative stability is much less complete than the qualitative picture described above. In the pure interaction case, Yan--Yao \cite{yan2022sharp} proved a quantitative stability which bounds the energy gap $\mathcal{F}(\omega) - \mathcal{F}(\mu)$ below by $W_{2}(\omega, \mu)^2$. In \cite{carlen2025stability}, Carlen--Lewin--Lieb--Seiringe established a sharp quantitative Lane-Emden inequality, and provided the energy deficit of aggregation-diffusion is bounded a discrepancy related to $L^{m/2}(\mathbb{R}^d)$ norm for $m \le 2$.

\subsection{Main results}
We focus on the quantitative stability of steady states of aggregation-diffusion energy~\eqref{0}. Let $\mu$ be the minimizer of the aggregation-diffusion energy~\eqref{0} centered at $0$, and $R > 0$ such that $\mathrm{supp}(\mu) \subseteq B_{R}(0)$. Following Yan--Yao \cite{yan2022sharp}, we define.
\begin{equation*}
    \Lambda_{R} = \min_{0 \le r \le 2R} \frac{W^{\prime}_{\mathrm{ra}}(r)}{r},
\end{equation*}
as the \textit{least attraction} constant, where $W_{\mathrm{ra}}$ is the radial profile of the kernel $W$ defined such that $W_{\mathrm{ra}}(x) = W(|x|)$. Since $W_{\mathrm{ra}}$ is increasing and non-degenerate at origin, the constant $\Lambda_{R}$ is always strictly positive. For Riesz kernels $W(x) = |x|^k/k$ where $-d < k < 0$, we have $\Lambda_R = (2R)^{k-2}$.

We first establish the following transport inequality, a variational characterization of transversal coercivity in Wasserstein spaces.
\begin{theorem}[transport inequality]\label{theorem:1}
For all $\omega$ supported in $B_{R}(x_{\omega})$ for some $x_{\omega}$ in $\mathbb{R}^{d}$, we have
\begin{equation*}
    \mathcal{F}(\omega) - \mathcal{F}(\mu) \ge \Lambda_{R} \inf_{a \in \mathbb{R}^{d}}\frac{ W_{2}(\tau_{a}\omega, \mu)^{2}}{4},
\end{equation*}
where $\tau_{a}$ is the translation by vector $a$.
\end{theorem}

We then provide a spectral lower-bound of the Wasserstein Hessian at equilibrium $\mu$, in analogous with~\eqref{1-2-transport}. Locating the spectrum of Wasserstein Hessian would be important for quantitative long-time convergence of~\eqref{0-1}, through the principle of linearization \cite{denzler2005fast, seis2014long}.%\footnote{In these two works, one can explicitly calculate the eigenvalues and eigenfunctions of Wasserstein Hessian based on the explicit form of Barenblatt profiles.}
\begin{theorem}[spectrum]\label{theorem:2_spectrum}There exists a unique self-adjoint operator $D^{2}_{W}\mathcal{F}(\mu)$ on $T_{\mu}\mathscr{P}_{2}(\mathbb{R}^{d})$ with
\begin{equation*}
    \frac{\mathrm{d}^{2}}{\mathrm{d}t^2}\Big|_{t = 0} \mathcal{F}\big((\mathrm{id} + tX)\# \mu\big) = \int_{\mathbb{R}^d}\big(X, D^{2}_{W}\mathcal{F}(\mu) X\big) \ud \mu
\end{equation*}
for all smooth vector fields $X$, and its spectrum is located as follows,
\begin{equation*}
    \sigma(D_{W}^{2}\mathcal{F}(\mu))
    \subseteq
    \{0\} \cup [\Lambda_{R}, +\infty).
\end{equation*}
The kernel of $D_{W}^{2}\mathcal{F}(\mu)$ is $d$-dimensional, which encodes the invariance under translation on $\mathbb{R}^d$. In other words, we have the following Poincar\'e inequality
\begin{equation*}
    \int_{\mathbb{R}^d}\big(X, D^{2}_{W}\mathcal{F}(\mu) X\big) \ud \mu \ge \Lambda_R \|X\|^{2}_{L^{2}(\mathbb{R}^{d},\, \mu)},\quad \text{ for all } \int_{\mathbb{R}^{d}} X\ud \mu = 0.
\end{equation*}
\end{theorem}

In the forthcoming work \cite{deng2026sharp}, the author have proved a exponential relaxation towards equilibrium for special compactly supported aggregation-diffusion equations~\eqref{0-1} with a suitable class of sublinear kernels. The method is the principle of linearization based on this spectral gap theorem. Compared to Carrillo–McCann–Villani \cite{carrillo2003kinetic}, we only characterized the Wasserstein Hessian at minimizer instead of a general measure. Hence the consideration of initial value is limited. %It is unclear to the author whether the spectrum of Wasserstein Hessian will be well-behaved in a local neighborhood of $\mu$. %In the current construction, we heavily rely on $\mu$ being radial decreasing, which is necessary for applying the height interpolation of Delgadino--Yan--Yao \cite{delgadino2022uniqueness}, and estimate high-order eigenvalues based on Yan--Yao \cite{yan2022sharp}.

For the future direction, it is interesting to investigate the spectrum of Wasserstein Hessian for a general compactly supported measure $\omega$ with non-degeneracy at support boundary. It would also be important to obtain Lojasiewicz gradient type inequalities (as that in~\eqref{1-2-transport}):
\begin{equation*}
    \mathcal{F}(\omega) - \mathcal{F}(\mu) \lesssim \big\|D_{W}\mathcal{F}(\omega)\big\|_{L^{2}(\mathbb{R}^{d},\, \omega)}^{2}
\end{equation*}
where $\mu$ is the minimizer, or an inequality related to Lyapunov-Schmidt reduction \cite[Chapter 3.12]{simon2012theorems}
\begin{equation*}
    \mathcal{F}(\omega) - \mathcal{F}(\omega_{\sharp}) \lesssim \big\|D_{W}\mathcal{F}(\omega)\big\|_{L^{2}(\mathbb{R}^{d},\, \omega)}^{2}
\end{equation*}
where $\omega_{\sharp}$ is the radial rearrangement of $\omega$ (see Section~\ref{sec:2-2-rearrange}). These would be critical for analyzing the long-time behavior of the aggregation-diffusion equation~\eqref{0-1}.

The proof of Theorem~\ref{theorem:1} is given in Section~\ref{sec:3-transport}, and the proof of Theorem~\ref{theorem:2_spectrum} is given in Section~\ref{sec:4-spec}.

\section*{Acknowledgment}
\textbf{AI usage}: 
\begin{enumerate}
    \item Lemma~\ref{lemma:7} comes from a discussion with ChatGPT 5.6 Sol, which is used to bridge the height interpolation to Wasserstein geodesics, and is the core technical difficulty of the work.
    \item All the numerical results and illustrations are conducted purely by ChatGPT 5.6 Sol.
\end{enumerate}
No AI is used in writing. The author takes full responsibility for the content.

\section{Preliminary}

\subsection{Optimal transport}\label{sec:2-1-optimal}
In this part, we collect some discussion on the optimal transport theory. Our references are \cite{ambrosio2008gradient, villani2008optimal, santambrogio2015optimal}. We define the second-momentum of a measure,
\begin{equation*}
    M_{2}(\omega) := \int_{\mathbb{R}^d} |x|^2 \ud \omega,
\end{equation*}
let $\mathscr{P}_{2}(\mathbb{R}^{d})$ be the space of probability measures with finite $M_2$, and $\mathscr{P}_{2, \mathrm{ac}}(\mathbb{R}^{d})$ be those measures that are absolutely continuous with respect to the Lebesgue measure. The set $\mathscr{P}_{2}(\mathbb{R}^{d})$ is a complete metric space with a natural metric $W_{2}$ from optimal transport \cite[Chapter 6]{villani2008optimal}.

\begin{definition}We define the Wasserstein distance $W_{2}$ between $\mu$ and $\omega$ in $\mathscr{P}_{2}(\mathbb{R}^{d})$ as
\begin{equation*}
    W_{2}(\mu, \omega) = \inf \bigg\{ \int_{\mathbb{R}^{d}} |S(x) - x|^2 \ud \mu : S\#\mu = \omega\bigg\}^{\frac{1}{2}}
\end{equation*}
where $S$ is minimized among measurable maps from $\mathbb{R}^{d}$ to $\mathbb{R}^{d}$, and $\#$ is the push-forward operator defined by $(S\# \mu)[A] = \mu[S^{-1}(A)]$ for all Borel measurable set, or can be defined weakly by
\begin{equation*}
    \int_{\mathbb{R}^d} \varphi \ud (S \# \mu) = \int_{\mathbb{R}^d} \varphi \circ S \ud \mu
\end{equation*}
for all bounded continuous function $\varphi$. The minimizer $S_{\ast}$ is called the optimal transport map.
\end{definition}

The following polar factorization theorem of Y. Brenier \cite{brenier1991polar} characterizes the most important geometry of Wasserstein distance.
\begin{lemma}[polar factorization]\label{lemma-3-brenier}Given any $\mu \in \mathscr{P}_{2, \mathrm{ac}}(\mathbb{R}^{d})$ and $\omega \in \mathscr{P}_{2}(\mathbb{R}^{d})$, there exists a unique convex function $\phi_{\mu \to \omega}$ called Brenier potential such that $\nabla \phi_{\mu \to \omega}$ is the optimal transport map from $\mu$ to $\omega$. On the other hand, if one can find a convex function $\phi$ with $\nabla \phi \# \mu = \omega$, we have $\phi = \phi_{\mu \to \omega}$.
\end{lemma}

This polar factorization theorem captures the local geometric structure of $\mathscr{P}_{2}(\mathbb{R}^{d})$. The following construction is originally due to Otto \cite{otto2001geometry}, and the rigorous version of the tangent space is given in \cite{ambrosio2008gradient}.
\begin{definition}\label{def_4_tangent}We define the tangent space $T_{\mu} \mathscr{P}_{2}(\mathbb{R}^{d})$ as the closure of the space of smooth and compactly supported gradient vector fields under the $L^{2}(\mathbb{R}^d,\mu)$ norm:
\begin{equation*}
    T_{\mu} \mathscr{P}_{2}(\mathbb{R}^{d}) = \overline{\big\{\nabla \phi : \phi \in C_{c}^{\infty}(\mathbb{R}^d)\big\}}_{L^{2}(\mathbb{R}^d,\,\mu)}
\end{equation*}
We define the Otto metric $g_{\mu}$ as a bilinear form on the tangent space $T_{\mu} \mathscr{P}_{2}(\mathbb{R}^d)$ given by
\begin{equation*}
    g_{\mu}(X, Z) := \int_{\mathbb{R}^d} X \cdot Z \ud \mu.
\end{equation*}
For notational simplicity, we denote the Otto metric as $(X, Z)_{\mu}$.
\end{definition}

\subsection{Rearrangement}\label{sec:2-2-rearrange}
In this part, we collect some discussion on rearrangement of mass densities. Our references are \cite{carrillo2016nonlinear, delgadino2022uniqueness, yan2022sharp}.
\begin{definition}[radial symmetric]
For given non-negative measurable function $f$, we say it is \textit{radial symmetric}, if there exists a function $f_{\mathrm{ra}} : [0, \infty) \to \mathbb{R}$ such that $f(x) = f_{\mathrm{ra}}(|x|)$. Furthermore, we say it is \textit{radial decreasing}, if in addition $f_{\mathrm{ra}}$ is right-continuous and decreasing.
\end{definition}

We next recall the Schwarz spherical decreasing rearrangement, which provides a canonical ways of constructing radial decreasing functions for given density distribution $f$.
\begin{definition}[rearrangement]We say $f_{\sharp}$ is the Schwarz spherical decreasing rearrangement of $f$, if 
\begin{equation*}
    f_{\sharp}(x) = f_{\mathrm{Har}}(\omega_{d}|x|^d)
\end{equation*}
where $\omega_{d}$ is the volume of $d$-dimensional unit ball, and $f_{\mathrm{Har}}$ is the Hardy-Littlewood one-dimensional decreasing rearrangement of $f$ defined as the right inverse of the accumulation function
\begin{equation*}
    f_{\mathrm{Har}}(s) = F^{-1}(s),\quad \text{ where } F(t) = \big|\{x : f(x) \ge t\}\big|,
\end{equation*}
where $F(t)$ is the Lebesgue measure of the $t$-upper level set of $f$.
\end{definition}

One property of this Schwarz spherical decreasing rearrangement (or we call the radial symmetric rearrangement) is the preservation of Lebesgue norm given as follows.
\begin{lemma}For any non-negative $f$ in $L^{m}(\mathbb{R}^{d})$, we have
\begin{equation*}
    \|f_{\sharp}\|_{L^{m}(\mathbb{R}^{d})} = \|f\|_{L^{m}(\mathbb{R}^{d})},\quad \text{ for all } 1\le m \le \infty,
\end{equation*}
where $f_{\sharp}$ is the Schwarz spherical decreasing rearrangement.
\end{lemma}

We recall the following sharp second-momentum stability of Yan--Yao \cite[Proposition 4.1]{yan2022sharp}.
\begin{lemma}\label{lemma-5-redution}
For any $\omega\in\mathscr P_2(\mathbb{R}^d)$, let $\omega_{\sharp}$ be its radial symmetric rearrangement, we have
\begin{equation*}
    M_2(\omega) - M_2(\omega_{\sharp}) \ge W_2(\omega, \omega_{\sharp})^2,
\end{equation*}
where $M_{2}(\omega) := \int_{\mathbb{R}^d} |x|^2 \ud \omega$.
\end{lemma}

\subsection{Height interpolation}
We recall the height interpolation technique between radial decreasing densities introduced in Delgadino--Yan--Yao \cite{delgadino2022uniqueness}. We first define the height function and radius function, illustrated in Figure~\ref{fig:1}.
\begin{definition}[height function]Given bounded radial decreasing density $\varrho : \mathbb{R}^{d} \to [0, \infty)$, we let
\begin{equation*}
    H[\varrho] : [0, 1] \to \big[0, \|\varrho\|_{L^{\infty}}\big]
\end{equation*}
be its height function, defined such that
\begin{equation*}
    \int_{\mathbb{R}^{d}} \min\big\{\varrho(x), H[\varrho](s)\big\} \ud x = s.
\end{equation*}
In other words, $H[\varrho](s)$ is the height of the plane such that the mass below the plane equals to $s$.
\end{definition}

\begin{definition}\label{def:6:radius}
We define the radius function
\begin{equation*}
    R[\varrho](s) = \big(\omega_{d} H[\varrho]^{\prime}(s)\big)^{-\frac{1}{d}},
\end{equation*}
which is the radius of the level set of the profile $\varrho$ at height $H[\varrho](s)$.
\end{definition}

The following lemma due to \cite[Lemma 2.1]{delgadino2022uniqueness} establishes the correspondence between radial decreasing densities and hight function.
\begin{lemma}Given $\varrho$ compactly supported, continuous, and strictly decreasing within support interior, its height function $H[\varrho]$ is $C^1$, increasing and strictly convex.

Given $C^1$, increasing and strictly convex function $H[\varrho]$, one can recover the $\varrho$ as follows
\begin{equation*}
    \varrho(x) = \int_{0}^{1} \mathbf{1}\big[B_{R[\varrho](s)}\big](x) H[\varrho]^{\prime}(s) \ud s,
\end{equation*}
where $\mathbf{1}[A]$ is the characteristic function of the set $A$.
\end{lemma}

Height interpolation is a continous curve $\varrho(t)$ in $\mathscr{P}_{2}(\mathbb{R}^d)$ between two given endpoints such that the corresponding height function $[\varrho(t)]$ is affine is the sense of Lebesgue functions.
\begin{definition}[height interpolation]Given bounded radial decreasing densities $\varrho_{0}$ and $\varrho_{1}$, we define the height interpolation $\varrho(t)$ from $\varrho_{0}$ to $\varrho_{1}$ such that
\begin{equation*}
    H[\varrho(t)] = (1 - t)H[\varrho_{0}] + t H[\varrho_{1}].
\end{equation*}
\end{definition}

The following lemma introduced in \cite[Proposition 2.2]{delgadino2022uniqueness} gives the regularity of height interpolation in the Wasserstein space.
\begin{lemma}\label{lemma_8_height}Given compactly supported densities $\mu, \omega$ in $\mathscr{P}_2(\mathbb{R}^d) \cap C(\mathbb{R}^d)$, that are strictly decreasing within the support. Let $\varrho(t)$ be the height interpolation from $\mu$ to $\omega$.

Then $\varrho(t)$ is a Lipschitz curve in the Wasserstein space $(\mathscr{P}_2(\mathbb{R}^d), W_2)$. In addition, there exists a bounded radial vector field $Z(x, t)$ such that
\begin{equation*}
    \frac{\partial \varrho(t)}{\partial t} + \mathrm{div}(\varrho(t)Z(x, t)) = 0
\end{equation*}
in the sense of distribution.
\end{lemma}

\begin{figure}
    \centering
    \includegraphics[width=\linewidth]{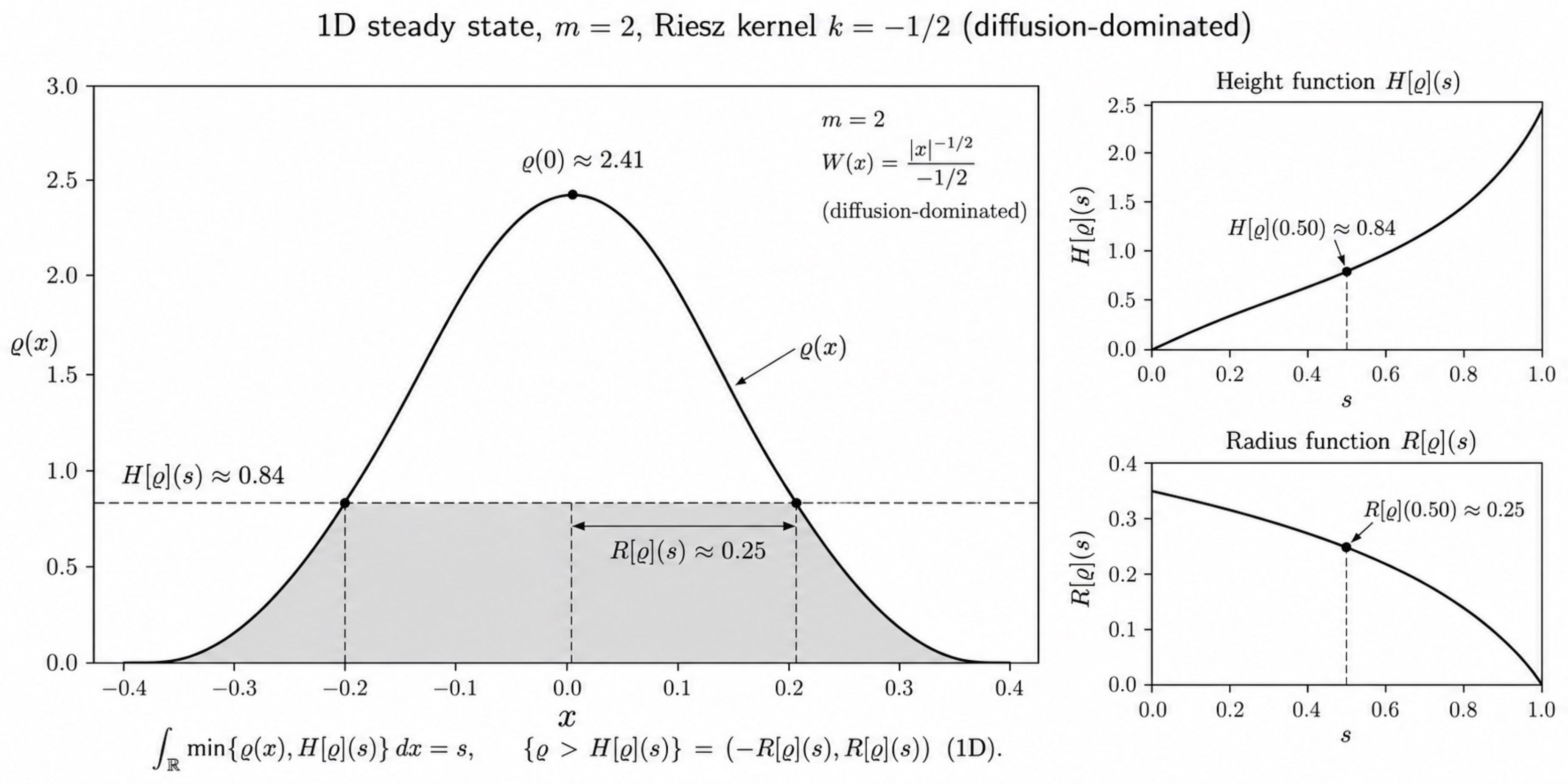}
    \caption{\small{Illustration of the height function and radius function.}}
    \label{fig:1}
\end{figure}

\section{Transport inequality for radial profiles}\label{sec:3-transport}
In this section, we construct transport inequality shown in Theorem~\ref{theorem:1}.

\subsection{From height to transport}
In this part, we introduce a technical lemma that bridges height interpolation to the Wasserstein geodesics interpolation, which is one of the technical novelties we employed. We let
\begin{equation*}
    \mathcal{I}_A(\varrho) := \frac12 \int_{\mathbb{R}^d} |x|^2 \varrho(x) \ud x.
\end{equation*}
By a slight abuse of notation, we denote
\begin{equation*}
    \mathcal{I}_{A}(H[\varrho]) = \mathcal{I}_{A}(\varrho),
\end{equation*}
where the left hand side refers to $\mathcal{I}_{A}$ as a functional over height functions. This energy functional $\mathcal{I}_{A}$ will serve as a comparison energy for the aggregation part free energy functional~\eqref{0}, which will be used in the next subsection for providing a lower bound on the second variation of $\mathcal{F}(\varrho)$.

Define the Bregman divergence
\begin{equation*}
    D_{\mathcal{I}, A}(H, H_{0}) := \mathcal{I}_A(H) - \mathcal{I}_A(H_{0}) - D\mathcal{I}_A(H_{0})(H - H_{0}),
\end{equation*}
we introduce the following lemma.
\begin{lemma}\label{lemma:7}
Let \(\omega,\mu\) be radial decreasing densities, we have
\begin{equation*}
    D_{\mathcal{I}, A}(H[\omega], H[\mu]) \ge \frac12 W_2(\omega, \mu)^{2}.
\end{equation*}
\end{lemma}

\begin{proof}
For notational simplicity, we denote $H(s)$ and $R(s)$ the height function and radius function of $\varrho$ at level $0\le s \le 1$. By tranforming the computation into polar coordinate, one has
\begin{equation*}
    \begin{aligned}
        \mathcal{I}_{A}(\varrho) &= \frac{1}{2} \int_{\mathbb{R}^d} \int_{0}^{1} |x|^2 \mathbf{1}[B_{R(s)}](x) H^{\prime}(s) \ud s \ud x\\
        &= \frac{d\omega_d}{2} \int_{0}^{1} \int_{0}^{R(s)}  r^2 H^{\prime}(s) \times r^{d-1}\ud r \ud s = \frac{d\omega_d }{2(d + 2)} \int_{0}^{1} R(s)^{2 + d} \times H^{\prime}(s) \ud s
    \end{aligned}
\end{equation*}
where $\omega_{d}$ is the volume of unit ball. Since $\omega_{d}H^{\prime} = R^{-d}$, we can get
\begin{equation*}
    \mathcal{I}_{A}(H) = \frac{d}{2(d + 2)} \int_{0}^{1} R(s)^2 \ud s = \frac{d}{2(d + 2)} (\omega_d)^{-2/d} \int_{0}^{1} H^{\prime}(s)^{-2/d} \ud s,
\end{equation*}
which shows that $\mathcal{I}_{A}$ can be seen as a well-defined functional over the height function. Now, we let $Q_{\omega} = H[\omega]^{\prime}$ and $Q_{\mu} = H[\mu]^{\prime}$. By the definition of Bregman divergence, one can get
\begin{equation*}
    D_{\mathcal{I}, A}(H[\omega], H[\mu]) = \frac{d}{2(d+2)} (\omega_d)^{-2/d}
    \int_0^1 Q_{\omega}^{-2/d} - Q_{\mu}^{-2/d}
       +\frac2d Q_{\mu}^{-2/d-1}(Q_{\omega} - Q_{\mu}) \ud s
\end{equation*}
Applying the following elementary inequality for $x = Q_{\omega}/Q_{\mu}$,
\begin{equation*}
    x^{-\frac{2}{d}} - 1 + \frac2d(x-1) \ge (x^{-\frac{1}{d}} - 1)^2,
\end{equation*}
we are able to get
\begin{equation*}
    D_{\mathcal{I}, A}(H[\omega], H[\mu]) \ge \frac{d}{2(d+2)} \int_0^1 \big(R[\omega](s) - R[\mu](s)\big)^2 \ud s.
\end{equation*}
In the following, we prove that 
\begin{equation*}
    \frac{d}{d + 2}\int_0^1 \big(R[\omega](s) - R[\mu](s)\big)^2 \ud s \ge W_{2}(\omega, \mu)^2,
\end{equation*}
which will conclude the proof.

In fact, let
\begin{equation*}
    \omega_{s} := U_{R[\omega](s)},\quad \mu_{s} := U_{R[\mu](s)}
\end{equation*}
be the uniform distributions on $B_{R[\omega](s)}$ and $B_{R[\mu](s)}$ respectively, which are the slices of the probability distributions at height level $H(s)$. For these simple uniform distributions, one can explicitly compute the Wasserstein distance between them:
\begin{equation*}
    W_2(\omega_{s}, \mu_s)^2 = \frac{d}{d+2}\big(R[\omega](s) - R[\mu](s)\big)^2,
\end{equation*}
because the scaling map
\begin{equation*}
    S: x \mapsto R[\mu](s)^{-1}R[\omega](s)x
\end{equation*}
is the optimal transport map from $\mu_s$ to $\omega_{s}$ by Lemma~\ref{lemma-3-brenier}. In other words, the height interpolation coincides with the Wasserstein geodesic interpolation between uniform distributions.

Finally, because $\omega = \int_{0}^{1}\omega_s \ud s$ and $\mu = \int_{0}^1\mu_s \ud s$, we have
\begin{equation*}
    W_{2}(\omega, \mu)^2 \le \int_{0}^{1} W_{2}\big(\omega_s, \mu_s\big)^2 \ud s = \frac{d}{d+2} \int_{0}^{1} \big(R[\omega](s) - R[\mu](s)\big)^2 \ud s,
\end{equation*}
according to \cite[Theorem 4.8]{villani2008optimal}.
\end{proof}

\subsection{Radial case}In this part, we prove the transport inequality between radial decreasing densities. Our main result of this part is the following.
\begin{proposition}\label{prop_10-raidal}For any radial decreasing density $\omega$ supported in $B_{2R}(0)$, we have
\begin{equation*}
    \mathcal{F}(\omega) - \mathcal{F}(\mu) \ge \frac{\Lambda_R}{2} W_{2}(\omega, \mu)^{2},
\end{equation*}
where $\Lambda_{R} := \inf_{0 \le r \le 2R} W_{\mathrm{ra}}^{\prime}(r)/r$.
\end{proposition}

\begin{proof}
We split the energy into two parts,
\begin{equation*}
    \mathcal{F}(\varrho) = \mathcal{S}(\varrho) + \mathcal{I}(\varrho),\quad \mathcal{S}(\varrho) := \frac{1}{m-1} \int_{\mathbb{R}^d}\varrho^m\ud x \text{ and } \mathcal{I}(\varrho) := \frac12
    \int_{\mathbb{R}^d} (W \ast \varrho) \times \varrho  \ud x.
\end{equation*}
Let $\varrho(t)$ be the height interpolation from $\mu$ to $\omega$. Specifically, when both $\mu$ and $\omega$ are supported in $B_{R}(0)$, $\varrho(t)$ is supported in $B_{R}(0)$ for all $0 < t < 1$ \cite[Lemma 4.1]{delgadino2022uniqueness}. Our method of estimation is similar to \cite{carrillo2003kinetic}, where one combines a qualitative convexity of internal energy part $\mathcal{S}(\varrho(t))$, and the quantitative convexity comes from the internal energy $\mathcal{I}(\varrho(t))$.

In particular, for the internal energy, it is convex under height interpolation \cite[Proposition 2.3]{delgadino2022uniqueness},
\begin{equation}\label{2-10-s}
    \frac{\mathrm{d}^{2} \mathcal{S}(\varrho(t))}{\mathrm{d}t^{2}} \ge 0.
\end{equation}
For the interaction energy, we introduce the comparison kernel $A(x) = \tfrac{1}{2}|x|^2$, and decompose $W$ according to the work of Yan--Yao \cite{yan2022sharp} as follows,
\begin{equation*}
    W(x) = \Lambda_{R} \times A(x) + B(x),
\end{equation*}
where $\Lambda_{R}$ is the least attraction constant. For kernel $A$, one has
\begin{equation*}
    \begin{aligned}
        \mathcal{I}_{A}(\varrho) &= \frac{1}{2}\int_{\mathbb{R}^d} (A \ast \varrho) \times \varrho \ud x = \frac{1}{4}\int_{\mathbb{R}^d \times \mathbb{R}^d} |x|^2 + |\underline{x}|^2 - 2 x \cdot \underline{x} \ud \varrho(x) \otimes \varrho(\underline{x})\\
        &= \frac{1}{2} \int_{\mathbb{R}^d} |x|^2 \ud \varrho - \bigg[\frac{1}{2} \int_{\mathbb{R}^d} x \ud \varrho\bigg]^2\\
        &= \frac{1}{2} \int_{\mathbb{R}^d} |x|^2 \ud \varrho,
    \end{aligned}
\end{equation*}
The decomposition is illustrated in Figure~\ref{fig:2}. We leverage the height interpolation \cite{delgadino2022uniqueness} to filter out the non-convex part $B$, and then examine the relationship between the height interpolation and Wasserstein geodesic interpolation on the kernel $A$

We let $B_{\mathrm{ra}}$ be the radial profile, defined such that $B_{\mathrm{ra}}(|x|) = B(x)$, by the definition of the least attraction constant $\Lambda_{R} = \min_{0 \le r \le 2R} W^{\prime}_{\mathrm{ra}}(r)/r$, one has
\begin{equation*}
    B^{\prime}_{\mathrm{ra}}(r) = W^{\prime}_{\mathrm{ra}}(r) - \Lambda_{R}r > 0,\quad \text{ for all } r \le 2R.
\end{equation*}
Hence $B$ obeys Assumption~\ref{assumption} (the kernel $B$ outside $B_{2R}$ does not matter because all the considered densities $\varrho$ are supported in $B_{R}$). By \cite[Proposition 2.4]{delgadino2022uniqueness}, we have
\begin{equation*}
    \frac{\mathrm{d}^{2} \mathcal{I}_{B}(\varrho(t))}{\mathrm{d}t^{2}} \ge 0.
\end{equation*}
Therefore, we have a lower bound for the second variation of interaction energy given as follows
\begin{equation}\label{2-10-i}
    \frac{\mathrm{d}^{2} \mathcal{I}(\varrho(t))}{\mathrm{d}t^{2}} = \Lambda_{R}\frac{\mathrm{d}^{2} \mathcal{I}_{A}(\varrho(t))}{\mathrm{d}t^{2}} + \frac{\mathrm{d}^{2} \mathcal{I}_{B}(\varrho(t))}{\mathrm{d}t^{2}} \ge \Lambda_{R} \frac{\mathrm{d}^2 \mathcal{I}_{A}(\varrho(t))}{\mathrm{d}t^2}.
\end{equation}
Since $\mu$ is the minimizer, we have $D_{W} \mathcal{F}(\mu) = 0$. As $\varrho(t)$ is a Lipschitz continuous connecting $\mu$ to $\omega$, by using the standard second-variation formula, and inequalities~\eqref{2-10-s} and~\eqref{2-10-i}
\begin{equation*}
    \begin{aligned}
        \mathcal{F}(\omega) - \mathcal{F}(\mu) &= \int_{0}^{1} (1 - t) \frac{\mathrm{d}^{2} \mathcal{F}(\varrho(t))}{\mathrm{d}t^{2}} \ud t\\
        &= \int_{0}^{1} (1 - t) \bigg[\frac{\mathrm{d}^{2} \mathcal{S}(\varrho(t))}{\mathrm{d}t^{2}} + \frac{\mathrm{d}^{2} \mathcal{I}(\varrho(t))}{\mathrm{d}t^{2}}\bigg] \ud t \ge \Lambda_{R} \int_{0}^{1} (1 - t) \frac{\mathrm{d}^{2} \mathcal{I}_{A}(\varrho(t))}{\mathrm{d}t^{2}} \ud t.
    \end{aligned}
\end{equation*}
Finally, because the height interpolation is the affine interpolation of height functions, we have
\begin{equation*}
    \begin{aligned}
        \Lambda_{R} \int_{0}^{1} (1 - t) \frac{\mathrm{d}^{2} \mathcal{I}_{A}(\varrho(t))}{\mathrm{d}t^{2}} \ud t &= \Lambda_{R}\times \bigg[\mathcal{I}_A(\omega) - \mathcal{I}_A(\mu) - \frac{\mathrm{d}}{\mathrm{d}t}\Big|_{t = 0} \mathcal{I}_A(\varrho(t))\bigg]\\
        &= \Lambda_{R}\times D_{\mathcal{I}, A}(H[\omega], H[\mu])
    \end{aligned}
\end{equation*}
where $D_{\mathcal{I}, A}(\omega, \mu)$ is the Begamn divergence given in Lemma~\ref{lemma:7}. From Lemma~\ref{lemma:7},
\begin{equation*}
    D_{\mathcal{I}, A}(\omega, \mu) \ge \frac{1}{2}W_{2}(\omega, \mu)^{2},
\end{equation*}
which concludes the proof.
\end{proof}

\begin{figure}[t]
    \centering
    \includegraphics[width=0.75\linewidth]{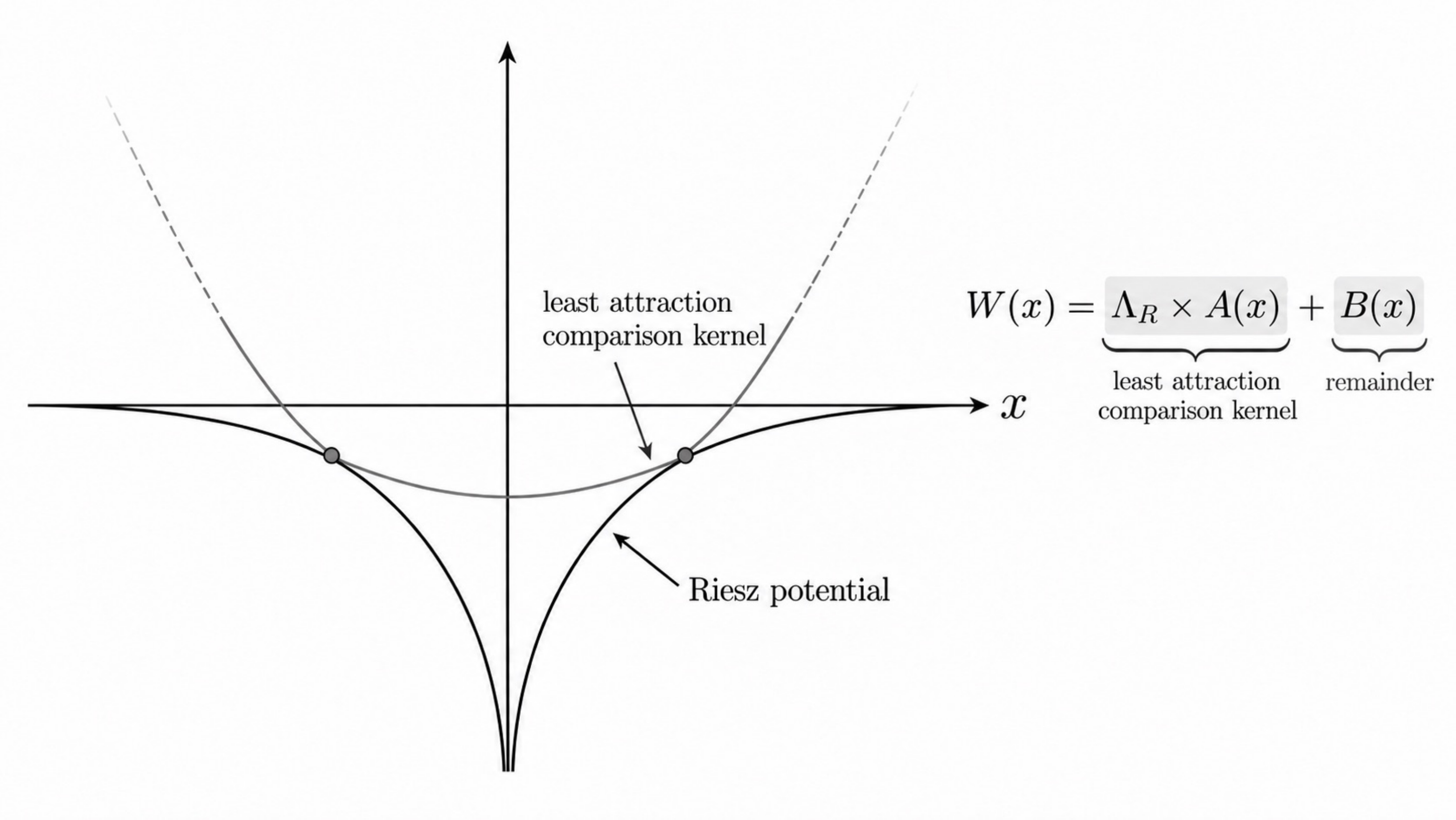}
    \caption{\small{Illustration of the kernel decomposition for a Reisz kernel $W(x) = -2|x|^{-\frac{1}{2}}$. By using comparison principle of ODE, the gradient flow on the Reisz kernel (which converges to $0$ in finite time) is at least faster than the least attraction comparison kernel in the region between contact points. The heuristics of the construction is that the attraction effect of the Reisz kernel is bounded below by that of the least attraction kernel.}}
    \label{fig:2}
\end{figure}

\subsection{From radial to general}
In this section, we extend the transport inequality between radial decreasing densities in Proposition~\ref{prop_10-raidal} to the general case where $\omega$ is compactly supported. Building on the quantitative stability of Yan--Yao \cite{yan2022sharp}, we claim the following reduction inequality.
\begin{proposition}\label{eq:nonradial-stability}For any $\omega$ supported in $B_{R}(x_{\omega})$ for some $x_{\omega}$ in $\mathbb{R}^d$, let $\omega_{\sharp}$ be is radial decreasing rearrangement, we have
\begin{equation*}
    \mathcal{F}(\omega) - \mathcal{F}(\omega_{\sharp}) \ge \frac{\Lambda_{R}}{2} \inf_{a \in \mathbb{R}^{d}} W_2(\tau_{a}\omega, \omega_{\sharp})^2.
\end{equation*}
where $\Lambda_{R} := \inf_{0 \le r \le 2R} W_{\mathrm{ra}}^{\prime}(r)/r$.
\end{proposition}

\begin{proof}
Since rearrangement preserves every $L^{m}(\mathbb{R}^{d})$ norm, $\mathcal{S}(\omega)=\mathcal{S}(\omega_{\sharp})$. The interaction part
\begin{equation*}
    \mathcal{I}(\omega) - \mathcal{I}(\omega_{\sharp}) \ge \frac{\Lambda_{R}}{2} \inf_{a \in \mathbb{R}^{d}} W_2(\tau_{a}\omega, \omega_{\sharp})^2.
\end{equation*}
was proven by Yan--Yan \cite[Theorem 1.2]{yan2022sharp} based on Lemma~\ref{lemma-5-redution}.
\end{proof}

\begin{proof}[Proof of Theorem~\ref{theorem:1}]
The idea is L. Simon's splitting method \cite[Chapter 3.12]{simon2012theorems} in his generalization of Lojasiewicz inequalities. We reduce $\omega$ into the sub-manifold of radial decreasing densities via radial rearrangement $\omega_{\sharp}$, and examine the relation between $\omega_\sharp$ and $\mu$, whose geometric structures are relatively simple.

In particular, we split the energy gap:
\begin{equation}\label{3-5-splitting}
    \mathcal{F}(\omega) -\mathcal{F}(\mu) = \underbrace{(\mathcal{F}(\omega) - \mathcal{F}(\omega_{\sharp}))}_{\text{reduction}} + \underbrace{(\mathcal{F}(\omega_{\sharp}) -\mathcal{F}(\mu))}_{\text{radial profile}}.
\end{equation}
For the reduction part, we apply Proposition~\ref{eq:nonradial-stability}, and for the radial profile, we apply Proposition~\ref{prop_10-raidal}. By doing these, we have
\begin{equation*}
    \mathcal{F}(\omega) -\mathcal{F}(\mu) \ge \frac{\Lambda_{R}}{2} \inf_{a \in \mathbb{R}^{d}} W_{2}(\tau_{a} \omega, \omega_{\sharp})^{2} + \frac{\Lambda_{R}}{2} W_{2}(\omega_{\sharp}, \mu)^{2}.
\end{equation*}
By $2a^2 + 2b^2 \ge (a + b)^2$ and triangle inequality, we can get
\begin{equation*}
    \mathcal{F}(\omega) -\mathcal{F}(\mu) \ge \frac{\Lambda_{R}}{4} \inf_{a \in \mathbb{R}^{d}} \big(W_{2}(\tau_{a} \omega, \omega_{\sharp}) + W_{2}(\omega_{\sharp}, \mu)\big)^{2} \ge \frac{\Lambda_{R}}{4} \inf_{a \in \mathbb{R}^{d}} W_{2}(\tau_{a} \omega, \mu)^{2},
\end{equation*}
which is the transport inequality we claim.
\end{proof}

\section{Spectrum of the Wasserstein Hessian}\label{sec:4-spec}
Our goal in this section is to leverage an orthogonality between the direction of reduction $\omega \to \omega_{\sharp}$ and the submanifold of radial decreasing densities to improve the estimate in~\eqref{3-5-splitting} and then locate the spectrum of the Wasserstein Hessian expressed by the Poincar\'e inequality
\begin{equation*}
    (X, D^{2}_{W}\mathcal{F}(\mu) X)_\mu \ge \Lambda_R \|X\|^{2}_{L^{2}(\mathbb{R}^{d},\, \mu)},
\end{equation*}
for all $\int_{\mathbb{R}^{d}} X\ud \mu = 0$ in Theorem~\ref{theorem:2_spectrum}. The computation of Wasserstein Hessian is given in the following proposition. For its derivation, we refer the reader to \cite[Section 3]{carrillo2003kinetic}.
\begin{proposition}\label{propo-12-wass}In the tangent space $T_{\mu}\mathscr{P}_2(\mathbb{R}^d)$, we have
\begin{equation*}
    (X, D^{2}_{W}\mathcal{F}(\mu) X)_{\mu} = (X, D^{2}_{W}\mathcal{S}(\mu) X)_{\mu} + (X, D^{2}_{W}\mathcal{I}(\mu) X)_{\mu},
\end{equation*}
where $\mathcal{S}(\mu)$ and $\mathcal{I}(\mu)$ are the internal and interaction energies, and
\begin{equation*}
    \begin{aligned}
        (X, D^{2}_{W}\mathcal{S}(\mu) X)_{\mu} &= \int_{\mathbb{R}^{d}} \|\nabla X\|^{2}_{\mathrm{HS}} + (m - 1) (\mathrm{div} X)^{2} \ud \mu^{m}\\
        (X, D^{2}_{W}\mathcal{I}(\mu) X)_{\mu} &= \int_{\mathbb{R}^{d} \times \mathbb{R}^{d}} \frac{1}{2}\delta_{y \to x} X \cdot \big(D^{2}W(x - y) \delta_{y \to x} X\big) \ud \mu \otimes \mu,
    \end{aligned}
\end{equation*}
where $\|\cdot\|_{\mathrm{HS}}$ is the Hilbert-Schmidt norm of a matrix, also known as the Frobenius norm, given by squared average of the matrix coefficients, and $\delta_{y \to x} X := X(x) - X(y)$ denotes the misalignment of the velocity at $x$ and $y$.
\end{proposition}

Before start, we establish some functional analysis lemmas. Let $\Omega$ be the support of the minimizer $\mu$, we define degenerate weighted Sobolev space.
\begin{equation*}
    H^{1}(\Omega, \mu) := \bigg\{\phi \text{ in } L^{1}_{\mathrm{loc}}(\Omega) : \int_{\Omega} |\nabla \phi|^2 \ud \mu < \infty\bigg\}.
\end{equation*}
We let $H^{1}(\Omega, \mu)/\mathbb{R}$ where the quotient over $\mathbb{R}$ identifies the function up to an addition constant. It is naturally equipped with a norm
\begin{equation}\label{4-5-qo}
    \|\phi\|_{H^{1}(\Omega, \,\mu)/\mathbb{R}} = \|\phi - (\phi)_{\mu}\|_{L^{2}(\Omega,\,\mu)} + \|\nabla \phi\|_{L^{2}(\Omega,\,\mu)}
\end{equation}
where $(\phi)_\mu = \int_{\Omega} \phi \ud \mu$ is the weighted average of the function. We show the following.
\begin{lemma}\label{lemma-17-iso}
$H^{1}(\Omega, \mu)/\mathbb{R}$ is isomorphic to the tangent space $T_{\mu}\mathscr{P}_{2}(\mathbb{R}^{d})$ as Hilbert spaces, where the isomorphism is the gradient operator.
\end{lemma}
\begin{lemma}\label{lemma-18-dense}
The operator $D^{2}_{W}\mathcal{F}(\mu)$ is symmetric, nonnegative, and densely defined on $T_{\mu}\mathscr{P}_{2}(\mathbb{R}^d)$.
\end{lemma}
The proofs of Lemmas~\ref{lemma-17-iso} and~\ref{lemma-18-dense} are given in Section~\ref{sec:4-3-4}.
\begin{proof}[Proof of self-adjoint extension in Theorem~\ref{theorem:2_spectrum}]
The self-adjoint extension of $D_{W}^{2}\mathcal{F}(\mu)$ on $T_{\mu}\mathscr{P}_{2}(\mathbb{R}^{d})$ is the result of Friedrichs extension theorem, by Lemma~\ref{lemma-18-dense}.
\end{proof}

\subsection{Spherical harmonics}To locate the spectrum, we let $(r, \theta) = (|x|, x/|x|)$ be the polar coordinate, and
\begin{equation*}
    -\Delta_{\mathbb S^{d-1}}Y_{\ell, \nu}(\theta)
    =
    \lambda_\ell Y_{\ell, \nu}(\theta),
    \quad
    \lambda_\ell=\ell(\ell+d-2).
\end{equation*}
be the spherical harmonics where $1 \le \nu \le n_{\ell}$ and $n_{\ell}$ is the number of spherical harmonics in the $\ell$-th sector. Consider the spherical harmonic decomposition in polar coordinates,
\begin{equation*}
    \phi(r \theta) = \sum_{\ell = 0}^{\infty} \sum_{\nu = 1}^{n_\ell} \phi_{\ell, \nu}(r) Y_{\ell, \nu}(\theta),
\end{equation*}
and the radial coefficients $\phi_{\ell, \nu}(r) = \int_{\mathbb{S}^{d - 1}} \phi(r\theta) Y_{\ell, \nu}(\theta) \ud \mathrm{vol}_{\mathbb{S}^{d - 1}}$, and $\mathrm{vol}_{\mathbb{S}^{d - 1}}$ is the canonical surface measure on the sphere.

We define operator $H_\mu$ such that
\begin{equation*}
    (H_\mu\phi, \phi)_\mu = (D_{W}^2\mathcal{F}(\mu)\nabla \phi, \nabla \phi)_\mu
\end{equation*}
for all smooth functions $\phi$.
As steady state $\mu$ and kernel $W$ are radial symmetric (invariant under rotation), the operator $H_\mu$ commutes with rotational operators. Hence, there exists a sequence of radial operators $(H_{\mu}^{(\ell)})_{\ell = 1}^{\infty}$, and we can decompose the operator $H_\mu$ as follows,
\begin{equation}\label{5-10-b}
    (H_\mu\phi)(r\theta) = \sum_{\ell = 0}^{\infty} \sum_{\nu = 1}^{n_\ell}  (H_{\mu}^{(\ell)}\phi_{\ell, \nu})(r) Y_{\ell, \nu}(\theta).
\end{equation}
We define the following spherical harmonics sectors, which constitute orthogonal decomposition of the tangent space $T_{\mu}\mathscr{P}_2(\mathbb{R}^d)$.
\begin{definition}
We define the $\ell$-th spherical harmonics sector 
\begin{equation*}
    \mathscr{Y}_{\ell} = \overline{\big\{Z(x) = \nabla \big(\phi(x) Y_{\ell, \nu}(\theta)\big) \text{ where } \phi \text{ is smooth radial symmetric, and } 1 \le \nu \le n_{\ell}\big\}}_{L^{2}(\mathbb{R}^d,\, \mu)}.
\end{equation*}
where $x = r\theta$.
\end{definition}

\subsection{Ground-state}\label{sec:4-3-1}In view of the decomposition~\eqref{5-10-b}, it is sufficient to investigate the eigenvalues in each harmonics sector $\mathscr{Y}_{\ell}$.

For $\ell = 0$, spherical harmonics sector is given by
\begin{equation*}
    \mathscr{Y}_{\ell = 0} := \overline{\big\{Z(x) = \nabla \phi(x) \text{ where } \phi \text{ is smooth and radial symmetric}\big\}}_{L^{2}(\mathbb{R}^d,\, \mu)}.
\end{equation*}
This sector contributes to radial dilation. To characterize the eigenvalue, we leverage the following one-to-one correspondence between Wasserstein perturbation $Z$ and height perturbation $Z_H$, based on \cite[Equation (4.8)]{delgadino2022uniqueness} and Lemma~\ref{lemma_8_height}.
\begin{lemma}\label{lemma_10_corres}For $\mu$ obeying Theorem~\ref{proposition-3-steady}, and bounded smooth function $Z_{H} : [0, 1] \to \mathbb{R}$ with $Z_{H}(0) = 0$, there exists a height perturbation $\mu_{Z}(t)$ defined such that
\begin{equation*}
    H[\mu_{Z}(t)](s) = H[\mu](s) + t Z_{H}(s),
\end{equation*}
for small $t$. The vector field
\begin{equation*}
    Z(x) = -\frac{1}{d} \frac{Z_{H}(s)}{H[\mu](s)} x,
\end{equation*}
is smooth radial symmetric where $|x| = R[\varrho](s)$, and it corresponds to the metric derivative of $\mu_{Z}(t)$ at $t = 0$, or the initial velocity vector of the continuity equation in Lemma~\ref{lemma_8_height}.
\end{lemma}

By standard approximation on polar coordinates, one has the following.
\begin{lemma}\label{lemma_9_corres}The space of vector fields $\zeta(|x|)x$ where $\zeta$ is $C^{\infty}_c((0, \infty))$ is dense in the space of radial vector fields under $L^2(\mathbb{R}^d,\mu)$ topology, for radial symmetric and absolutely continuous $\mu$.
\end{lemma}

Given bounded smooth $Z_H$ with $Z_H(0) = 0$, let $\mu_{Z}(t)$ be the height perturbation under $Z_H$ as that in Lemma~\ref{lemma_10_corres}. As $\mu$ is non-degenerate at origin and near support boundary (from Theorem~\ref{proposition-3-steady}), the radial direvative of $\mu$ is non-zero, and $H[\mu]'$ is positive and uniformly bounded below from zero. Hence we have (from Definition~\ref{def:6:radius})
\begin{equation*}
    \begin{aligned}
        R[\mu_Z](s) &= (\omega_d)^{-1/d} (H[\mu_Z]'(s))^{-\frac{1}{d}}\\
        &= (\omega_d)^{-1/d} (H[\mu]'(s) + tZ_H^\prime(s))^{-\frac{1}{d}} \le R[\mu](s) + t\varepsilon_{Z}
    \end{aligned}
\end{equation*}
for all bounded smooth $Z_{H}$, with a constant $\varepsilon_Z$ depending on $Z'_H$. It ensures that, when $t$ is small, $\mu_{Z}(t)$ remains in $B_R(0)$ where $B_R(0)$ is a ball containing the support of $\mu$ in its interior (We also refer the reader to \cite[Lemma 4.1]{delgadino2022uniqueness} for proving that $\mu_Z(t)$ is supported within $B_R$).

By applying the transport inequality of radial profiles (Proposition~\ref{prop_10-raidal}) within the ball $B_R$, we are able to get:
\begin{equation}\label{4-3-f}
    \frac{\mathcal{F}(\mu_{Z}(t)) - \mathcal{F}(\mu)}{\tfrac{1}{2}t^2} \ge \Lambda_{R} \frac{W_{2}(\mu_{Z}(t), \mu)^{2}}{t^2}.
\end{equation}
Building on Lemma~\ref{lemma_10_corres}, let $Z$ be the metric derivative of $\mu_{Z}(t)$ in $\mathscr{P}_{2}(\mathbb{R}^d)$, we can take limit in~\eqref{4-3-f} to get
\begin{equation}\label{4-4-f}
    \big(Z, D_W^2\mathcal{F}(\mu) Z\big)_{\mu} \ge \Lambda_{R} \|Z\|^2_{L^2(\mathbb{R}^d,\, \mu)}.
\end{equation}
Since $Z_H$ is arbitrary, Lemma~\ref{lemma_9_corres} indicates the variational characterization of eigenvalue~\eqref{4-4-f} can be made dense in $\mathscr{Y}_{\ell = 0}$. Hence, the ground state eigenvalue is located within $[\Lambda_R, \infty)$.

\subsection{Special first-order}\label{sec:4-3-2}
For $\ell = 1$, spherical harmonic sector is given by
\begin{equation*}
    \mathscr{Y}_{\ell = 1} := \overline{\big\{Z(x) = \nabla (\phi(x) \theta_{\nu}) \text{ where } \phi \text{ is smooth radial symmetric, and } 1 \le \nu \le d\big\}}_{L^{2}(\mathbb{R}^d,\, \mu)},
\end{equation*}
where $\theta_{\nu}$ is the $\nu$-th coordinate function of $\mathbb{R}^d$ restricted to $\mathbb{S}^{d-1}$. In the case $\phi(x) = c_\nu$ for some constant $c_\nu$, the vectors are $Z(x) =  c_\nu e_{\nu}$ for $1 \le \nu \le d$. They correspond to translation modes in coordinate directions. By the invariance of the aggregation-diffusion energy~\eqref{0} under translation,
\begin{equation*}
    \mathcal{F}((\mathrm{id} + t e_{\nu}) \# \mu) = \mathcal{F}(\mu) \text{ for all } t > 0.
\end{equation*}
Hence, for $Z(x) = \sum_{\nu = 1}^{d} c_\nu \nabla \theta_{\nu}$ with real coefficients $c_\nu$,
\begin{equation*}
    \lim_{t \to 0} \frac{\mathcal{F}((\mathrm{id} + tZ) \# \mu) - \mathcal{F}(\mu)}{t^2} = 0 = 0 \times \|Z\|^2_{L^2(\mathbb{R}^d,\, \mu)}.
\end{equation*}
It implies that the Wasserstein Hessian $D^{2}_W\mathcal{F}(\mu)$ has a kernel of dimension at least $d$. We will next show that this space spanned by coordinate vectors are exactly the kernel.

\subsection{Remainders from reduction}\label{sec:4-3-3}
For other general perturbations $X$, i.e., those orthogonal to $Z(x) = x_\nu$ for $\nu = 1, \cdots, d$ and $\mathscr{Y}_{\ell = 0}$, we use the reduction inequality in Proposition~\ref{eq:nonradial-stability} to estimate the spectrum.

We establish the following geometry lemma.
\begin{lemma}\label{lemma_18}Given smooth vector fields $Z$ such that
\begin{equation*}
    Z \perp \mathscr{Y}_{\ell = 0} \oplus \mathrm{span}_{1 \le \nu \le d} (\nabla x_{\nu}),
\end{equation*}
in the tangent space $T_\mu \mathscr{P}_2(\mathbb{R}^d)$. Let $\omega(t) = (\mathrm{id} + tX) \# \mu$, we have
\begin{equation*}
    \lim_{t \to 0} \frac{\inf_{a}W_{2}(\tau_{a}\omega(t), \omega(t)_{\sharp})^2}{t^2} = \|Z\|_{L^{2}(\mathbb{R}^d,\, \mu)}^{2},
\end{equation*}
where $\omega(t)_{\sharp}$ is the radial rearrangement of $\omega(t)$.
\end{lemma}

The proof will be given after spectral characterization. In particular, from reduction estimate in Proposition~\ref{eq:nonradial-stability}, we have
\begin{equation*}
    \mathcal{F}(\omega) - \mathcal{F}(\omega_{\sharp}) \ge \frac{\Lambda_{R}}{2} \inf_{a \in \mathbb{R}^{d}} W_2(\tau_{a}\omega, \omega_{\sharp})^2.
\end{equation*}
As $\mu$ is the minimizer, we have $\mathcal{F}(\omega(t)_{\sharp}) \ge \mathcal{F}(\mu)$, and therefore, we can deduce that
\begin{equation*}
    \mathcal{F}(\omega(t)) - \mathcal{F}(\mu) = \underbrace{\mathcal{F}(\omega(t)) - \mathcal{F}(\omega(t)_{\sharp})}_{\text{reduction}} + \underbrace{\mathcal{F}(\omega(t)_{\sharp}) - \mathcal{F}(\mu)}_{\text{non-negative}} \ge \inf_{a \in \mathbb{R}^d} W_{2}(\tau_a\omega(t), \omega(t)_{\sharp})^{2},
\end{equation*}
By Lemma~\ref{lemma_18}, we can take the limit $t \to 0$ to get
\begin{equation*}
    \begin{aligned}
        \big(Z, D^{2}_{W}\mathcal{F}(\mu) Z\big)_{L^{2}(\mathbb{R}^{d},\, \mu)} &= \lim_{t \to 0} \frac{\mathcal{F}(\omega(t)) - \mathcal{F}(\mu)}{\frac{1}{2}t^{2}}\\
        &\ge \Lambda_{R} \lim_{t \to 0} \frac{\inf_{a} W_{2}(\tau_a\omega(t), \omega(t)_{\sharp})^{2}}{t^{2}} = \Lambda_{R}\|X\|^{2}_{L^{2}(\mathbb{R}^{d},\, \mu)},
    \end{aligned}
\end{equation*}
which shows that the rest of the spectrum are located in $[\Lambda_{R}, \infty)$.

\begin{proof}[Proof of Lemma~\ref{lemma_18}]We first fix the center using the orthogonality with respect to first order spherical harmonics, and then prove the lemma based on the orthogonality with respect to radial variation modes.
~\\
\textit{Fixing the center}. We let $\phi_{\omega_{\sharp} \to \omega}$ be the Brenier potential from $\omega_{\sharp}$ to $\omega$ by Lemma~\ref{lemma-3-brenier}. Then, for any given displacement $a$ in $\mathbb{R}^d$, $\phi_{\omega_{\sharp} \to \omega}(x) + a\cdot x$ is a convex function with
\begin{equation*}
    \nabla \big(\phi_{\omega_{\sharp} \to \omega}(x) + a\cdot x\big) \# \omega_\sharp = \tau_{a} \omega.
\end{equation*}
It means that $\phi_{\omega_{\sharp} \to \omega}(x) + a\cdot x$ is the Brenier potential from $\omega_{\sharp}$ to $\tau_{a}\omega$. By using
\begin{equation*}
    |\nabla \phi + a - x|^2 = |\nabla \phi - x|^2 + 2a \cdot(\nabla \phi - x) + |a|^2,
\end{equation*}
one can get
\begin{equation*}
    \begin{aligned}
        W_{2}(\tau_{a} \omega(t), \omega(t)_{\sharp})^2 &= \int_{\mathbb{R}^d} |\nabla \phi_{\omega_{\sharp} \to \omega}(x) + a - x|^2 \ud \omega(t)_{\sharp}\\
        &= W_{2}(\omega(t), \omega(t)_{\sharp})^2 + 2a \bigg[\int_{\mathbb{R}^d} x \ud \omega(t) - \int_{\mathbb{R}^d} x \ud \omega(t)_{\sharp}\bigg] + |a|^2.
    \end{aligned}
\end{equation*}
By optimizing the vector $a$, one can get
\begin{equation}\label{5-10-a}
    \inf_{a\in\mathbb{R}^d} W_{2}(\tau_{a} \omega(t), \omega(t)_{\sharp})^2 = W_{2}(\omega(t), \omega(t)_{\sharp})^2 - \bigg[\int_{\mathbb{R}^d} x \ud \omega(t) - \int_{\mathbb{R}^d} x \ud \omega(t)_{\sharp}\bigg]^2
\end{equation}
Since $\omega(t)_{\sharp}$ is radial symmetric, we have $\int_{\mathbb{R}^d} x \ud \omega(t)_{\sharp} = 0$. On the other hand, by definition $\omega(t) = (\mathrm{id} + tZ)\# \mu$, we can then invoke the push-forward to get
\begin{equation*}
    \int_{\mathbb{R}^d} x \ud \omega(t) = \int_{\mathbb{R}^d} x + tZ(x) \ud \mu = 0
\end{equation*}
because $\mu$ is radial symmetric and $\int_{\mathbb{R}^d} Z \ud \mu = 0$ due to its orthogonality to $\mathrm{span}_{1 \le \nu \le d} (\nabla x_{\nu})$. Hence
\begin{equation*}
    \inf_{a\in\mathbb{R}^d} W_{2}(\tau_{a} \omega(t), \omega(t)_{\sharp})^2 = W_{2}(\omega(t), \omega(t)_{\sharp})^2.
\end{equation*}
according to~\eqref{5-10-a}.
~\\
\textit{Radial variation modes}. In the following, we show that
\begin{equation*}
    W_{2}(\omega(t)_{\sharp}, \mu) = o(t).
\end{equation*}
According to \cite[Lemma 2.10]{ambrosio2007gradient} and \cite{Antoine2018}, one has
\begin{equation*}
    \omega(t) = (\mathrm{id} + tZ)\# \mu = \mu - t \times \mathrm{div}(\mu Z) + o(t).
\end{equation*}
In particular, one can use the continuity equation as an infinitesimal push-forward \cite[Proposition 8.1.8]{ambrosio2008gradient}, or linearize the Jacobi formula of push-forward, to obtain this first-order asymptotics.

On the other hand, fixing a height level $\alpha$, let $R(\alpha)$ be the radius of the level set ball of $\mu$. Since $\mu$ is strictly radial decreasing, the $\alpha$-level set $\{x : \omega(x, t) = \alpha\}$ can be parameterized over a closed smooth curve
\begin{equation*}
    \zeta_{t, \alpha}(\theta) = R(\alpha)\theta + t\xi_\alpha(\theta) \theta + o(t)
\end{equation*}
when $t \le t_{\alpha}$ where $t_\alpha$ is a small real number depending on the height level $\alpha$, and $\xi_\alpha$ is a function on $\mathbb{S}^{d - 1}$, corresponding to the normal velocity of level set. As $\omega(\zeta_{t, \alpha}(\theta)) = \alpha$ is constant, differentiating it in time, we have
\begin{equation*}
    \begin{aligned}
        0 = \frac{\mathrm{d}}{\mathrm{d}t}\Big|_{t = 0} \omega(\zeta_{t, \alpha}(\theta), t) &= \frac{\mathrm{d}}{\mathrm{d}t}\Big|_{t = 0} \omega(R(\alpha) \theta, t) + \mu^{\prime}(R(\alpha) \theta) \frac{\mathrm{d}}{\mathrm{d}t}\Big|_{t = 0}\zeta_{t, \alpha}(\theta)\\
        &= -\mathrm{div}(\mu Z)(R(\alpha) \theta) + \mu^{\prime}(R(\alpha) \theta)\xi_{\alpha}(\theta),
    \end{aligned}
\end{equation*}
where $\mu^{\prime}$ is the derivative in the radial direction. Hence the variation of the level set (on the regular height level of $\mu$ where $\mu' \neq 0$) is
\begin{equation}\label{5-5-level}
    \xi_{\alpha}(\theta) = \frac{\mathrm{div}(\mu(x) Z(x))}{\mu^{\prime}(x)} = Z(x) \cdot \theta + \frac{\alpha}{\mu^{\prime}(x)} \mathrm{div}Z(x),\quad \text{ where } x = R(\alpha) \theta.
\end{equation}
The first term represents transport by $Z$, and the second represents the effect of local compression or expansion of the density.

By using the first variation formula of volume (see for example \cite[Proposition 1.19]{ritore2023isoperimetric}, where the variation of volume along vector fields is the integration of normal part of the vector fields on the boundary), we have
\begin{equation*}
    \frac{\mathrm{d}}{\mathrm{d}t}\Big|_{t = 0} \big|\{x : \omega(x, t) \ge \alpha\}\big| = R(\alpha)^{d - 1} \int_{\mathbb{S}^{d - 1}} \xi_{\alpha}(\theta) \ud \mathrm{vol}_{\mathbb{S}^{d - 1}}.
\end{equation*}
Here $|\cdot|$ is the Lebesgue measure, and the factor $R(\alpha)^{d - 1}$ comes from the scaling $\partial B_{R(\alpha)}$ into $\mathbb{S}^{d - 1}$. Similar method also applies to the level set variation of rearranged density $\omega_\sharp$, where we can get
\begin{equation*}
    \frac{\mathrm{d}}{\mathrm{d}t}\Big|_{t = 0} \big|\{x : \omega(x, t)_{\sharp} > \alpha\}\big| = R(\alpha)^{d - 1} \int_{\mathbb{S}^{d - 1}} \xi_{\alpha, \sharp}(\theta) \ud \mathrm{vol}_{\mathbb{S}^{d - 1}},
\end{equation*}
where $\xi_{\alpha, \sharp}$ is a constant by the radial symmetry of $\omega(t)_{\sharp}$. From the definition of radial rearrangement in Section~\ref{sec:2-2-rearrange}, we have
\begin{equation*}
    \frac{\mathrm{d}}{\mathrm{d}t}\Big|_{t = 0} \big|\{x : \omega(x, t)_{\sharp} \ge \alpha\}\big| = \frac{\mathrm{d}}{\mathrm{d}t}\Big|_{t = 0} \big|\{x : \omega(x, t) \ge \alpha\}\big|.
\end{equation*}
It implies that
\begin{equation*}
    \xi_{\alpha, \sharp} = \frac{1}{d\omega_d}\int_{\mathbb{S}^{d - 1}} \xi_\alpha(\theta) \ud \mathrm{vol}_{\mathbb{S}^{d - 1}},
\end{equation*}
where $d\omega_d$ is the surface area of $\mathbb{S}^{d - 1}$.

By using divergence formula in spherical coordinates (see for example, \cite[Chapter 1]{dai2013approximation})
\begin{equation*}
    \mathrm{div}(\mu Z) = \frac1{r^{d-1}}\partial_r(r^{d-1}\mu(r) Z_\mathrm{ra}) + \frac{\mu(r)}{r}\mathrm{div}_{\mathbb S^{d-1}}Z_\mathrm{tan}.
\end{equation*}
Here $Z_\mathrm{ra} = Z \cdot \theta$ and $Z_\mathrm{tan} = Z - Z_\mathrm{ra}$ are the normal and tangential parts of $Z$. From the definition of $\xi_{\alpha}$ in~\eqref{5-5-level}, we have
\begin{equation*}
    \begin{aligned}
        \int_{\mathbb{S}^{d - 1}} \xi_{\alpha}(\theta) \ud \mathrm{vol}_{\mathbb{S}^{d - 1}} &= \int_{\mathbb{S}^{d - 1}} \frac{\mathrm{div}(\mu Z)(R(\alpha)\theta)}{\mu^{\prime}(R(\alpha)\theta)} \ud \mathrm{vol}_{\mathbb{S}^{d - 1}}\\
        &= \frac{1}{R(\alpha)^{d - 1}\mu'(R(\alpha))} \partial_r\bigg[R(\alpha)^{d - 1} \mu(R(\alpha)) \int_{\mathbb{S}^{d - 1}} Z_{\mathrm{ra}}(R(\alpha)\theta) \ud \theta \bigg]\\
        &\quad + \frac{\mu(R(\alpha))}{R(\alpha)} \int_{\mathbb{S}^{d - 1}}\mathrm{div}_{\mathbb S^{d-1}}Z_\mathrm{tan} \ud \mathrm{vol}_{\mathbb{S}^{d - 1}}.
    \end{aligned}
\end{equation*}
As $Z$ is orthogonal to $\mathscr{Y}_{\ell = 0}$, $Z$ contains no zeroth order spherical harmonics, we have
\begin{equation*}
    \int_{\mathbb{S}^{d - 1}} Z_\mathrm{ra}(R(\alpha)\theta) \ud \theta = 0.
\end{equation*}
On the other hand, the vector field $Z_\mathrm{tan}$ contributes to tangential motions, which are reparameterization of the sphere. By the divergence theorem on the sphere, we have
\begin{equation*}
    \int_{\mathbb{S}^{d - 1}} \operatorname{div}_{\mathbb S^{d-1}}Z_\mathrm{tan}(R(\alpha)\theta) \ud \theta = 0.
\end{equation*}
By combining the above two facts, we are able to get
\begin{equation*}
    \xi_{\alpha, \sharp} = \frac{1}{d\omega_d}\int_{\mathbb{S}^{d - 1}} \xi_\alpha(\theta) \ud \mathrm{vol}_{\mathbb{S}^{d - 1}} = 0.
\end{equation*}
As $\omega_\sharp$ is radial symmetric and its first-order radial variation vanishes at $t = 0$, we have
\begin{equation*}
    \frac{\partial \omega(t)_{\sharp}}{\partial t}\Big|_{t = 0} = 0.
\end{equation*}
Therefore, the associated metric gradient of $\omega(t)_\sharp$ expressed in terms of the velocity vector in the continuity equation also vanishes identically in the first order at $t = 0$, which means that
\begin{equation*}
    |W_{2}(\omega(t), \omega(t)_{\sharp}) - W_{2}(\omega(t), \mu)| \le W_{2}(\omega(t)_{\sharp}, \mu) = o(t).
\end{equation*}
Taking the limit $t \to 0$, we can get
\begin{equation*}
    \lim_{t \to 0} \frac{\inf_{a} W_{2}(\tau_{a}\omega(t), \omega(t)_{\sharp})^2}{t^2} = \lim_{t \to 0}  \frac{W_{2}(\omega(t), \omega(t)_{\sharp})^2}{t^2} = \lim_{t \to 0} \frac{W_{2}(\omega(t), \mu)^2}{t^2} = \|Z\|_{L^{2}(\mathbb{R}^d,\, \mu)}^{2},
\end{equation*}
which concludes the proof.
\end{proof}

\subsection{Functional analysis lemmas}\label{sec:4-3-4}
In this part, we prove Lemmas~\ref{lemma-17-iso} and \ref{lemma-18-dense}.
\begin{proof}[Proof of Lemma~\ref{lemma-17-iso}]
For the gradient operator $\nabla$, one has
\begin{enumerate}
    \item (well-defined) If $\phi_0$ and $\phi_1$ are constant functions, then $\nabla (\phi_0 - \phi_1) = 0$, which means that $\nabla$ is well-defined on the quotient space.
    \item (injective) Since $\Omega$ is connected, then $\nabla \phi = 0$ implies that $\phi$ is constant almost everywhere.
    \item (surjective) By definition, $\nabla H^{1}(\Omega, \mu)$ is the space of $L^2$ gradient vector fields.
\end{enumerate}
These criteria show that $\nabla$ is a linear space isomorphism from $H^{1}(\mathbb{R}^d, \mu)/\mathbb{R}$ to $T_{\mu}\mathscr{P}_{2}(\mathbb{R}^d)$. We next prove that $\nabla$ is a Hilbert space isomorphism by showing
\begin{equation}\label{a-1-equiv}
    \|\nabla \phi\|_{\mu} \le \|\phi\|_{H^{1}(\Omega, \,\mu)/\mathbb{R}} \lesssim \|\nabla \phi\|_{\mu},
\end{equation}
where $\|\nabla \phi\|_{\mu}$ is the Otto metric on the tangent space, and $\|\phi\|_{H^{1}(\Omega, \,\mu)/\mathbb{R}}$ is the quotient space norm in~\eqref{4-5-qo}. Since $\mu$ is non-degenerate at support boundary from Theorem~\ref{proposition-3-steady}
\begin{equation*}
    \mu(x) \sim \delta(x, \mathbb{R} - \Omega)^{\frac{1}{m - 1}}.
\end{equation*}
One can use the Hardy-Poincar\'e inequality in \cite[Lemma 3 in Appendix]{seis2014long} and $\mu \le \|\mu\|_{L^{\infty}}^{m - 1} \mu^{2 - m}$ to show that
\begin{equation*}
    \int_{\Omega} |\phi - (\phi)_{\mu}|^2 \ud\mu \lesssim \int_{\Omega} |\phi - (\phi)_{\mu}|^2 \ud \mu^{2 - m} \lesssim \int_{\Omega} |\nabla \phi|^2 \ud \mu.
\end{equation*}
It then gives the norm equivalence as stated in~\eqref{a-1-equiv}.
\end{proof}

According to Seis \cite[Appendix]{seis2014long}, we have the following density based on the interior positivity of the weight $\mu$.
\begin{lemma}\label{lemma-13-density}For any $\phi$ in $H^{1}(\Omega, \mu)$, there exists a sequence of functions $(\phi_{\ell})$ in $C^{\infty}(\Omega)$ such that $\|\phi_{\ell} - \phi\|_{H^{1}(\Omega, \,\mu)} \to 0$ as $\ell \to \infty$.
\end{lemma}

\begin{proof}[Proof of Lemma~\ref{lemma-18-dense}]
Symmetry and non-negativity is the direct result of $D_{W}^{2}\mathcal{F}(\mu)$ being the Hessian of the minimizer. For the later, if $D_{W}^{2}\mathcal{F}(\mu)$ is not non-negativity, we can find $X$ such that 
\begin{equation*}
    (X, D_{W}^{2}\mathcal{F}(\mu)X)_{\mu} < 0,
\end{equation*}
and perpetuating $\mu$ in the $X$ direction will destroy the minimality of $\mu$, which leads to a contradiction. 

In the following, we focus on the Wasserstein Hessian $D_{W}^{2}\mathcal{F}(\mu)$ (given in Proposition~\ref{propo-12-wass}) being well defined for smooth gradient vector fields, regarding the density in Lemma~\ref{lemma-13-density}, and equivalence between $H^{1}(\mathbb{R}^d, \mu)/\mathbb{R}$ and $T_{\mu}\mathscr{P}_{2}(\mathbb{R}^d)$ as stated in Lemma~\ref{lemma-17-iso}. For bounded smooth vector fields, it is direct to see that the internal energy part
\begin{equation*}
    (X, D^{2}_{W}\mathcal{S}(\mu) X)_{\mu} = \int_{\mathbb{R}^{d}} \|\nabla X\|^{2}_{\mathrm{HS}} + (m - 1) (\mathrm{div} X)^{2} \ud \mu^{m}
\end{equation*}
is well-defined as $\mu^{m}$ is bounded.

We proceed to show that the interaction part $D_{W}^2\mathcal{I}(\mu)$ is well-defined for smooth vector fields under regularity assumption in Assumption~\ref{assumption-1}. Consider polar coordinate
\begin{equation*}
    r = |x - y| \text{ and }\theta = (x - y)/|x - y|,
\end{equation*}
we apply the gradient formula in spherical coordinates \cite[Chapter 1]{dai2013approximation} to show that
\begin{equation*}
    D^{2}W(x - y) = W_{\mathrm{ra}}^{\prime \prime}(r) (\theta \otimes \theta) + \frac{W_{\mathrm{ra}}^{\prime}(r)}{r} (\mathrm{id} - \theta \otimes \theta).
\end{equation*}
where $W_{\mathrm{ra}}$ is the radial profile of $W$ defined such that $W_{\mathrm{ra}}(|x|) = W(x)$. By leveraging this formula, we can get
\begin{equation*}
    (X, D^{2}_{W}\mathcal{I}(\mu) X)_{\mu} = \frac{1}{2} \int_{\mathbb{R}^{d} \times \mathbb{R}^{d}} W^{\prime \prime}_{\mathrm{ra}}(r) \times |\delta_{r} X|^{2} + \frac{W_{\mathrm{ra}}^{\prime}(r)}{r} \times |\delta_{\perp} X|^{2} \ud \mu \otimes \mu,
\end{equation*}
where $\delta_{r} X$ is the projection of $\delta_{y \to x} X$ onto the direction $\theta$, and $\delta_{\perp} X = X - \delta_{r} X$ is the orthogonal part. Based on the Lipschitz continuity of $X$ in the interior of $B_{R}$, we have
\begin{equation*}
    |\delta_{r} X|,\, |\delta_{\perp} X| \le |\delta_{y \to x} X| \lesssim |x - y| = r.
\end{equation*}
By using $0 < W_{\mathrm{ra}}^{\prime}(r) \lesssim r^{k - 1}$, $ |W_{\mathrm{ra}}^{\prime \prime}(r)| \lesssim r^{k - 2}$ from Assumptions~\ref{assumption} and~\ref{assumption-1}, we have
\begin{equation*}
    \begin{aligned}
        \big|\big(X, D^{2}_{W}\mathcal{I}(\mu) X\big)_{\mu}\big| &\lesssim \int_{\mathbb{R}^{d} \times \mathbb{R}^{d}} r^{2} |W^{\prime \prime}_{\mathrm{ra}}(r)| + r W_{\mathrm{ra}}^{\prime}(r) \ud \mu \otimes \mu\\
        &\lesssim \int_{B_{R} \times B_{R}} r^{k} \ud \mu \otimes \mu = \int_{B_{R} \times B_{R}} |x - y|^{k} \ud \mu \otimes \mu.
    \end{aligned}
\end{equation*}
As $\mu$ is bounded, we then apply H\"older's inequality to get
\begin{equation*}
    \int_{B_{R} \times B_{R}} |x - y|^{k} \ud \mu \otimes \mu \le \int_{B_{R}} |z|^{k} \ud z \times \|\mu\|_{L^{\infty}(B_{R})} \int_{B_{R}} \ud \mu  < \infty.
\end{equation*}
where the last inequality is given by the local integrability of the singularity $|x|^k$ for $-d < k$.
\end{proof}

\bibliographystyle{alpha}
%\bibliography{references}

\end{document}